\documentclass{scrartcl}

\usepackage{amsmath}
\usepackage{amsfonts}
\usepackage{amssymb}
\usepackage{amsthm}
\usepackage{dsfont}
\usepackage{enumerate}
\usepackage{url}
\usepackage[english]{babel}
\usepackage{graphicx}

\usepackage[format=plain,labelfont=bf,textfont=it]{caption}
\usepackage[numbers]{natbib}
\usepackage{cases}

\newtheorem{theorem}{Theorem}
\newtheorem{lemma}[theorem]{Lemma}
\newtheorem{proposition}[theorem]{Proposition}
\theoremstyle{definition}
\newtheorem*{remark}{Remark}
                       
\newcommand{\D}{\mathrm{d}}
\newcommand{\average}[1]{\left[ #1 \right]}
\newcommand{\cL}{\mathcal{L}}
\newcommand{\der}[2]{\frac{\partial #1}{\partial #2}}

\DeclareMathOperator{\EL}{EL}
\newcommand{\inner}[2]{\left\langle #1 \,, #2 \right\rangle}
\newcommand{\m}{\mathrm{mod}}
\newcommand{\cO}{\mathcal{O}}
\DeclareMathOperator{\sgn}{sgn}

\title{Numerical precession in variational discretizations of the Kepler problem}
\author{Mats Vermeeren}
\date{\normalsize \textit{Institut f\"ur Mathematik, MA 7-2, Technische Universit\"at Berlin, \\
Str.\@ des 17.\@ Juni 136, 10623 Berlin, Germany \\
E-mail:} \url{vermeeren@math.tu-berlin.de}}

\begin{document}

\maketitle

\begin{abstract}
\noindent
Kepler's first law states that the orbit of a point mass with negative energy in a classical gravitational potential is an ellipse with one of its foci at the gravitational center. In numerical simulations of this system one often observes a slight precession of the ellipse around the gravitational center. Using the Lagrangian structure of modified equations and a perturbative version of Noether's theorem, we provide leading order estimates of this precession for the implicit MidPoint rule (MP) and the St\"ormer-Verlet method (SV). Based on those estimates we construct some new numerical integrators that perform significantly better than MP and SV on the Kepler problem. 
\end{abstract}

\section{Introduction}

The Kepler problem models a point mass moving in a classical gravitational potential. Its Lagrangian is
\begin{equation}\label{eqsofmot}
\cL(x,\dot{x}) = \frac{1}{2} |\dot{x}|^2 + \frac{1}{|x|},
\end{equation}
where $|\cdot|$ denotes the Euclidean norm on $\mathbb{R}^N$. The equations of motion are 
\[ \ddot{x} = - \frac{x}{|x|^3}. \]
It is well known that the orbits of the Kepler problem with negative energy are ellipses with one of their foci at the origin. Since every orbit lies in a plane, it is sufficient to study this problem in $\mathbb{R}^2$.

In this work we are interested in numerical integration of the Kepler problem. Very good integrators for this problem are already available, see for example \cite{chin2007physics} and the references therein. Our main objective here is to illustrate methods to analyze and improve numerical integrators. For the sake of clarity we start from simple methods. Accordingly, the improved methods we construct will not be competitive compared with specialized methods available in the literature. 

Central in our treatment will be the \emph{precession} or \emph{perihelion advance} of the numerical orbits, i.e.\@ the slow rotation of the ellipse that the solution traces. For the exact solution there is no precession, but no common numerical method integrates the Kepler problem without precession. Using the theory of modified equations, we will provide leading order estimates of the precession for the St\"ormer-Verlet method and the implicit midpoint rule. We will use those estimates to construct some new methods which are superior for the Kepler problem. 

Throughout this paper we use the Lagrangian formulation of classical mechanics. We will describe the modified equations using modified Lagrangians and use a version of Noether's theorem to analyze the perturbation. We start by mentioning a few well-known properties of the Kepler problem that will be useful later on. 

\begin{proposition}
The angular momentum $\mathds{L} = x_1 \dot{x}_2 - \dot{x}_1 x_2$ and the total energy $\mathds{E} = \frac{1}{2} |\dot{x}|^2 - \frac{1}{|x|}$ are constants of motion of the Kepler problem in $\mathbb{R}^2$. Furthermore, the angular momentum satisfies
\[ \mathds{L}^2 = |x| |\dot{x}|^2 - \inner{x}{\dot{x}}^2, \]
where the brackets $\inner{\cdot}{\cdot}$ denote the standard scalar product on $\mathbb{R}^N$.
\end{proposition}

\begin{proposition}\label{prop-EandL}
Let $a$ and $b$ denote the semimajor and semiminor axes of an orbit respectively. Then
\begin{itemize}
\item the square of the angular momentum is $\mathds{L}^2 = \frac{b^2}{a}$,
\item the energy is $\mathds{E} = \frac{-1}{2a}$,
\item the period is $T = 2 \pi a^{3/2}$,
\item the eccentricity is $e = \sqrt{1-\frac{b^2}{a^2}}$.
\end{itemize}
\end{proposition}

A thorough analytical study of the Kepler problem, including proofs of these properties, can be found for example in \cite[Chapter 3]{goldstein1980classical}.

\section{Modified Lagrangians}\label{sect-old}

To study the behavior of a numerical method it is often useful to consider the modified equation, a perturbation of the original differential equation whose solutions interpolate the discrete solutions. Generally, modified equations are found as formal power series in the step size of the method. Here we will truncate these power series after the first nontrivial term. For an introduction to this subject, see \cite[Chapter IX]{hairer2006geometric} and the references therein.

It is well-known that the modified equation of a symplectic integrator applied to a Hamiltonian system is again Hamiltonian. This means that the modified equation of a variational integrator applied to a Lagrangian system is Lagrangian as well. We will use a Lagrangian for the modified equation as the basis of our analysis. For its construction we refer to \cite{vermeeren2015modified}.

The modified equation of a numerical integrator for the Kepler problem describes a perturbed Kepler problem. Perturbed Kepler problems are very relevant in celestial mechanics. In particular, one of the classical tests of general relativity is that its perturbation in the Kepler potential accounts for the precession of the orbit of the planet Mercury \cite{will1981theory} (along with perturbations caused by the gravitational pull of the other planets). A Hamiltonian treatment of perturbed Kepler problems can be found for example in \cite{goldstein1980classical} or  \cite{chin2007physics}. Here we will work in the Lagrangian framework.

\subsection{St\"ormer-Verlet method}

The St\"ormer-Verlet (SV) discretization with step size $h$ of a second order differential equation 
$\ddot{x} = f(x)$ is
\[ \frac{x_{k+1} - 2x_k + x_{k-1}}{h^2} = f(x_k). \]
If $f(x) = -\frac{\D}{\D x}U(x)$, this is the discrete Euler-Lagrange equation for
\[ L_{SV}(x_k,x_{k+1}) = \frac{1}{2} \left|\frac{x_{k+1}-x_k}{h}\right|^2 - \frac{1}{2} U(x_k) - \frac{1}{2} U(x_{k+1}). \]
As shown in \cite{vermeeren2015modified}, the modified Lagrangian of second order accuracy is
\begin{align*}
\cL_{\m,2}(x,\dot{x}) 
&=  \frac{1}{2} |\dot{x}|^2 - U(x) + \frac{h^2}{24} \Big( \inner{ U'(x) }{ U'(x) } - 2 \inner{ \dot{x} }{ U''(x) \dot{x} } \Big).
\end{align*} 
By definition its Euler-Lagrange equation agrees with the modified equation with a defect of order $\cO(h^4)$. In the particular case of the Kepler problem this becomes
\begin{equation}\label{modlag-sv}
\cL_{\m,2}(x,\dot{x}) 
= \frac{1}{2} |\dot{x}|^2 + \frac{1}{|x|} + \frac{h^2}{24} \left( \frac{1}{|x|^4} - 2 \frac{ |\dot{x}|^2 }{|x|^3} + 6 \frac{\inner{ x }{ \dot{x} }^2}{|x|^5} \right).
\end{equation}
A comparison of the numerical solution and the solution of the  modified equation of second order accuracy is shown in Figure\ \ref{fig-sv}.

\begin{figure}[ht]
\begin{minipage}{.49\linewidth}
\centering
\includegraphics[width=.9\linewidth]{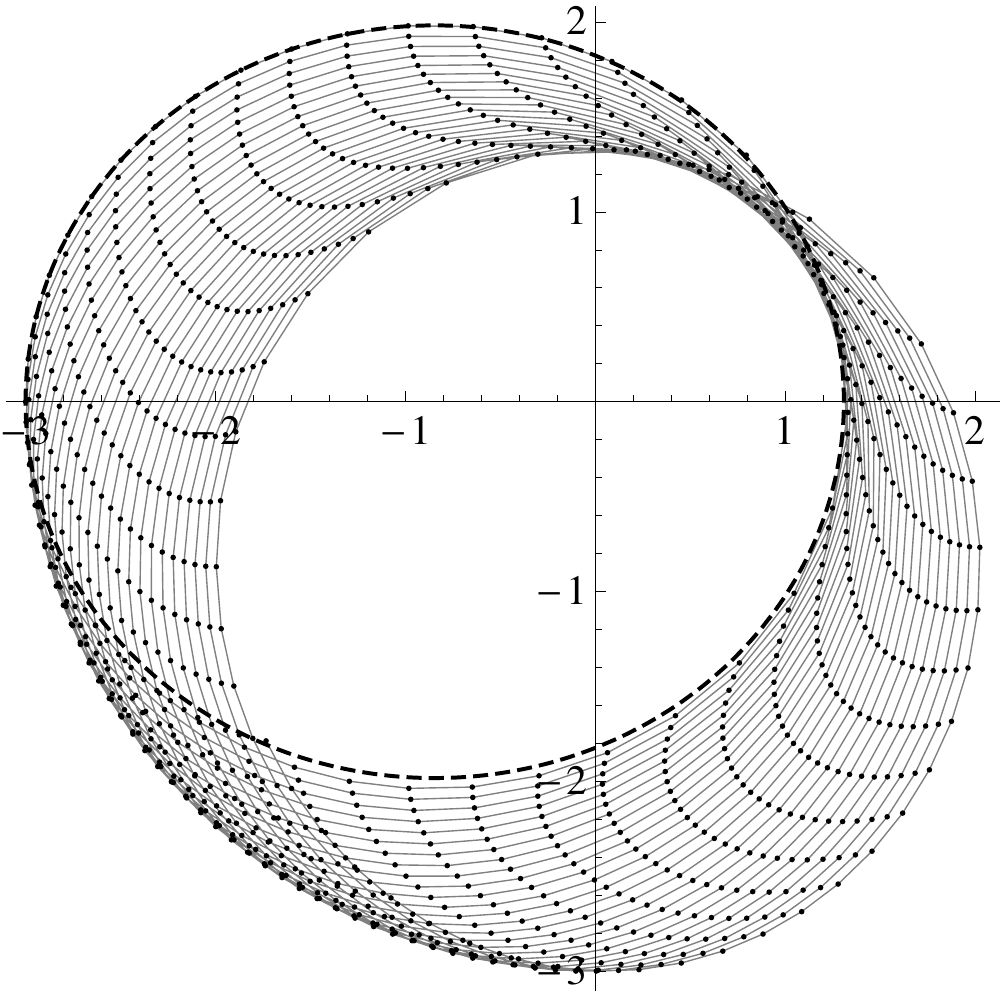}
\end{minipage}\hfill%
\begin{minipage}{.49\linewidth}
\centering
\includegraphics[width=.9\linewidth]{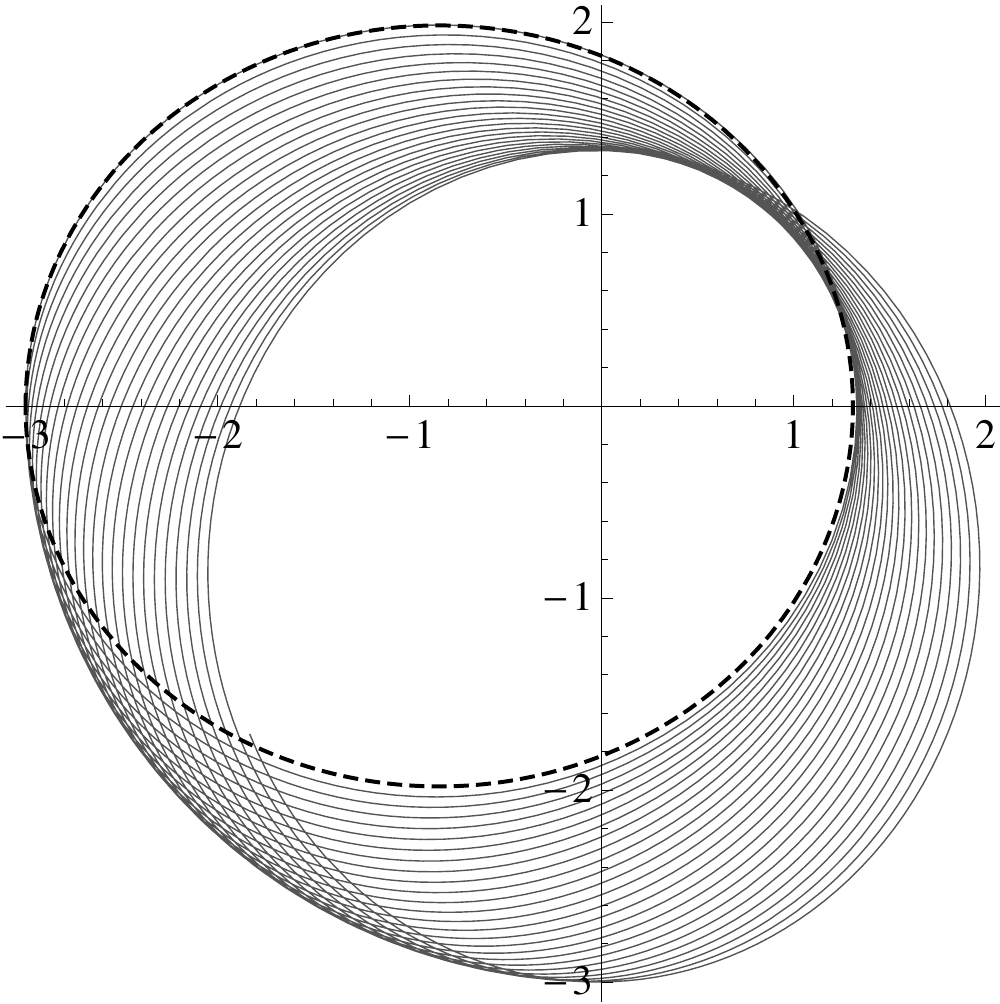}
\end{minipage}
\caption{St\"ormer-Verlet method with $1000$ steps of size $h=0.5$. Left: numerical solution. Right: modified equation of second order accuracy. In both images the dashed ellipse is the exact solution. The initial values are chosen as described in Section\ \ref{sect-init}.}\label{fig-sv}
\end{figure}

\subsection{Implicit midpoint rule}

The second order formulation of the implicit midpoint rule (MP) applied to the differential equation $\ddot{x} = f(x)$ is
\[ \frac{x_{k+1} - 2x_k + x_{k-1}}{h^2} = \frac{1}{2} f\left( \frac{x_k+x_{k+1}}{2} \right) + \frac{1}{2} f\left( \frac{x_{k-1}+x_k}{2} \right). \]
If $f(x) = -\frac{\D}{\D x}U(x)$, this is the discrete Euler-Lagrange equation for
\[ L_{MP}(x_k,x_{k+1}) = \frac{1}{2} \left| \frac{x_{k+1}-x_k}{h} \right|^2 - U\left(\frac{x_k + x_{k+1}}{2}\right). \]
The modified Lagrangian of second order accuracy is
\[ \cL_{\m,2} (x,\dot{x}) 
= \frac{1}{2} |\dot{x}|^2 + \frac{h^2}{24} \Big( \inner{ U'(x) }{ U'(x) } + \inner{ \dot{x} }{ U''(x) \dot{x} } \Big). \]
For the Kepler problem we have
\begin{equation}\label{modlag-mp}
\cL_{\m,2}(x,\dot{x}) 
= \frac{1}{2}|\dot{x}|^2 + \frac{1}{|x|} + \frac{h^2}{24} \left( \frac{1}{|x|^4} + \frac{|\dot{x}|^2}{|x|^3} - 3 \frac{\inner{ x }{ \dot{x} }^2}{|x|^5} \right). 
\end{equation}
A comparison of the numerical solution and the solution of the modified equation of second order accuracy is shown in Figure\ \ref{fig-mp}.

\begin{figure}[ht]
\begin{minipage}{.49\linewidth}
\centering
\includegraphics[width=.9\linewidth]{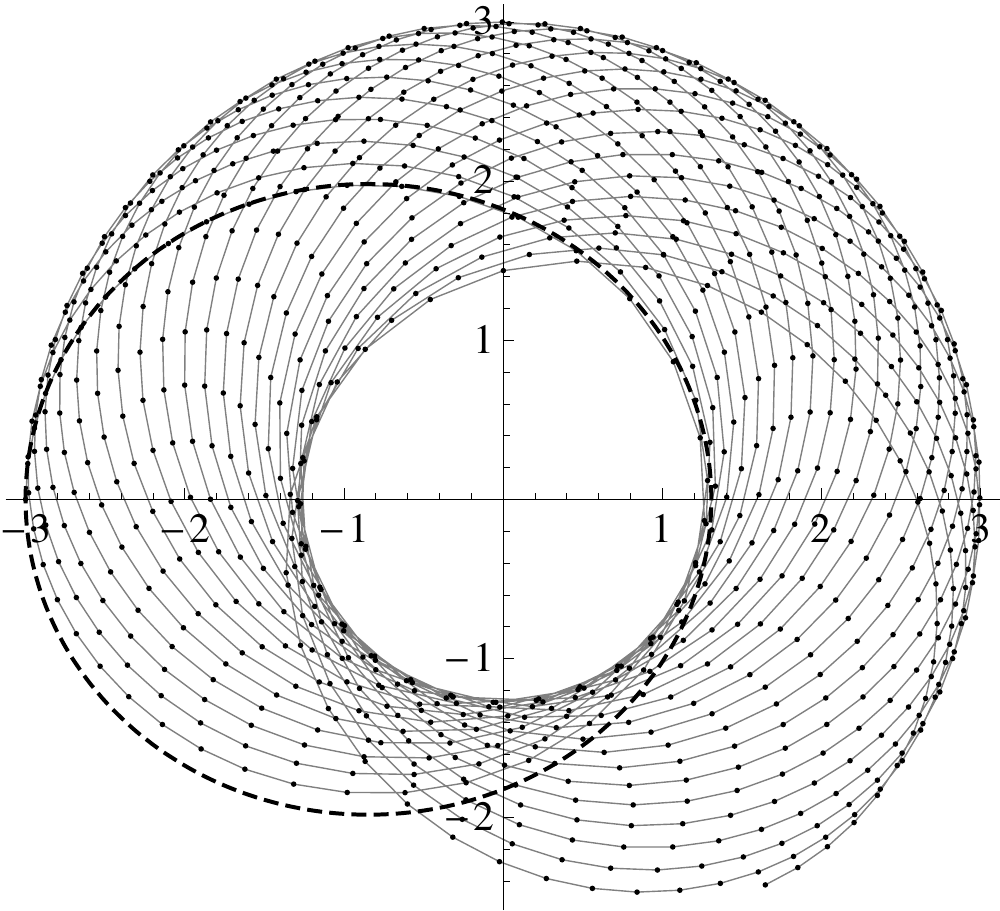}
\end{minipage}%
\hfill
\begin{minipage}{.49\linewidth}
\centering
\includegraphics[width=.9\linewidth]{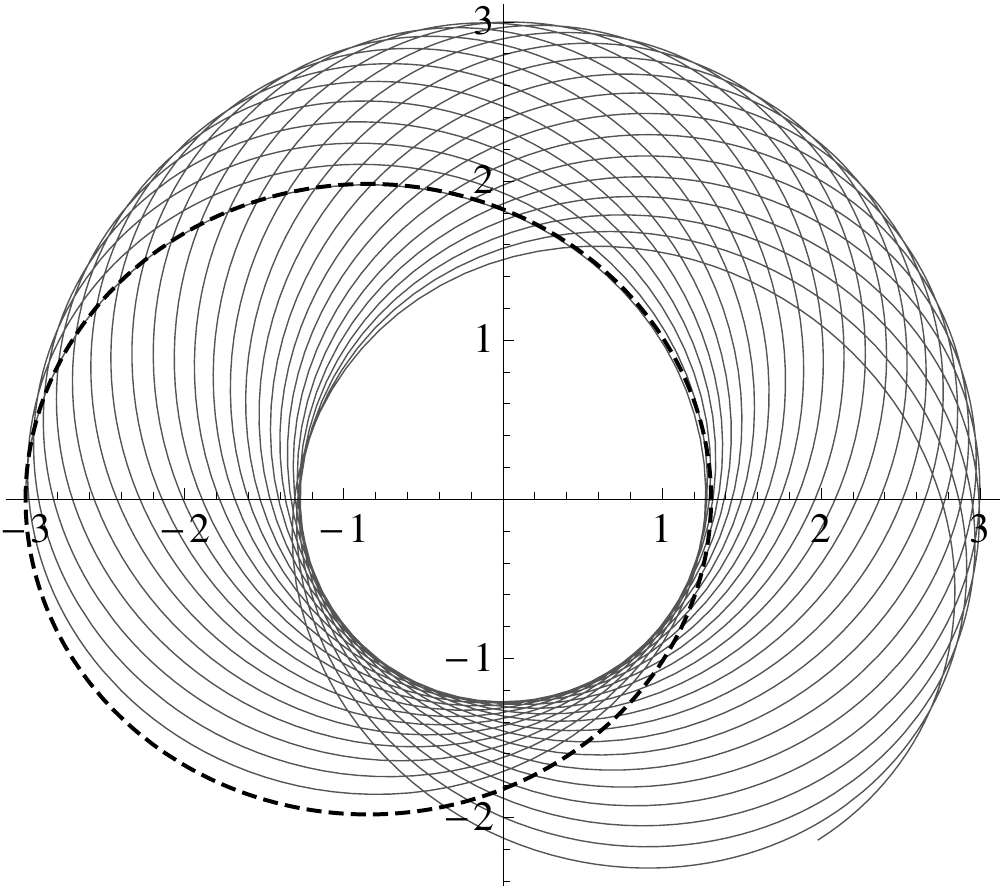}
\end{minipage}%
\caption{Implicit midpoint rule with $1000$ steps of size $h=0.5$. Left: numerical solution. Right: modified equation of second order accuracy. In both images the dashed ellipse is the exact solution. The initial values are chosen as described in Section\ \ref{sect-init}.}\label{fig-mp}
\end{figure}

\section{Noether's Theorem with perturbations}

The key observation in our study of the perturbed Kepler problem is that Noether's theorem \cite{noether1918invariante,olver2000applications} can be extended to describe how perturbations affect conserved quantities.

\begin{theorem}\label{thm-noether}
Consider a Lagrange function $\cL : T\mathbb{R}^2 \rightarrow \mathbb{R}$ and a horizontal vector field $\xi$ on $T\mathbb{R}^2$, i.e.\@ $\xi = \xi_1 \der{}{x_1} + \xi_2 \der{}{x_2}$ with coefficients $\xi_i$ that are functions $T\mathbb{R}^2 \rightarrow \mathbb{R}$. Let
\[ \xi^{(1)} = \sum_{i=1}^2 \left( \xi_i \der{}{x_i} + \dot{\xi}_i \der{}{\dot{x}_i} \right) \]
be the first prolongation of $\xi$, evaluated on solutions of the Euler-Lagrange equations, i.e.\@ with
\[ \dot{\xi}_i = \inner{ \der{\xi_i}{x} }{ \dot{x} } + \inner{ \der{\xi_i}{\dot{x}} }{ \left(\der{^2 \cL}{\dot{x}^2}\right)^{-1} \left(\der{\cL}{x} - \der{^2 \cL}{x \partial \dot{x}} \dot{x} \right) } . \]
If
\[ \xi^{(1)} \cL = \inner{ \der{G}{x} }{ \dot{x} } + \varepsilon F \]
for some functions $F: T \mathbb{R}^2 \rightarrow \mathbb{R}$ and $G:  \mathbb{R}^2 \rightarrow \mathbb{R}$ and a (small) parameter $\varepsilon \in \mathbb{R}$, then on solutions of the Euler-Lagrange equations we have
\[ \frac{\D}{\D t}\left( \inner{ \der{\cL}{\dot{x}} }{ \xi } - G \right) = \varepsilon F, \]
where by abuse of notation $\xi = (\xi_1,\xi_2)$. In particular, if $\varepsilon F = 0$, we have a conserved quantity $A := \sum_i \der{\cL}{\dot{x}_i} \xi_i - G$.
\end{theorem}
\begin{proof}
We have
\begin{align*} 
\frac{\D}{\D t}\left( \inner{ \der{\cL}{\dot{x}} }{ \xi } - G \right) 
&= \inner{ \frac{\D}{\D t} \der{\cL}{\dot{x}} }{ \xi } + \left( \xi^{(1)} \cL - \inner{ \der{\cL}{x} }{ \xi } \right) - \inner{ \der{G}{x} }{ \dot{x} } \\
&= \xi^{(1)} \cL - \inner{ \der{G}{x} }{ \dot{x} } - \inner{ \der{\cL}{x} - \frac{\D}{\D t} \der{\cL}{\dot{x}} }{ \xi } = \varepsilon F . \qedhere
\end{align*}
\end{proof}

\subsection{The Laplace-Runge-Lenz vector}

Following \cite{levy1971conservation} we consider the Kepler problem and the vector field $\xi$ defined by
\begin{equation}\label{defxi}
\xi_1 = -\frac{1}{2} x_2 \dot{x}_2 \qquad \text{and} \qquad
\xi_2 = x_1 \dot{x}_2 - \frac{1}{2} \dot{x}_1 x_2.
\end{equation}
On solutions we have
\[
\dot{\xi}_1 = -\frac{1}{2} \dot{x}_2^2 + \frac{1}{2} \frac{x_2^2}{|x|^3}  \qquad \text{and} \qquad
\dot{\xi}_2 = \frac{1}{2} \dot{x}_1 \dot{x}_2 - \frac{1}{2} \frac{x_1 x_2}{|x|^3}.
\]
A straightforward calculation then shows that
\[
\xi^{(1)} \cL = \inner{\der{\cL}{x}}{\xi} + \inner{\der{\cL}{\dot{x}}}{\dot{\xi}}
= \frac{\dot{x}_1}{|x|} - \frac{\inner{x}{\dot{x}} x_1}{|x|^3} 
= \frac{\D}{\D t}\left(\frac{x_1}{|x|}\right).
\]
Hence we can apply the unperturbed Noether theorem (i.e.\@ $\varepsilon F = 0$) with $G(x) = \frac{x_1}{|x|}$ and find that
\[ A(x,\dot{x}) = - \dot{x}_1 x_2 \dot{x}_2 + x_1 \dot{x}_2^2 - \frac{x_1}{|x|} = |\dot{x}|^2 x_1 - \inner{x}{\dot{x}} \dot{x}_1 - \frac{x_1}{|x|} \]
is a conserved quantity. 

The conserved quantity $A$ is the first component of the Laplace-Runge-Lenz (LRL) vector, which points from the gravitational center to the aphelion and has a magnitude equal to the eccentricity $e$ of the orbit. The second component of the LRL vector is 
\[ B(x,\dot{x}) = |\dot{x}|^2 x_2 - \inner{x}{\dot{x}} \dot{x}_2 - \frac{x_2}{|x|} \]
and can be obtained by setting $\xi_1 = x_2 \dot{x}_1 - \frac{1}{2} x_1 \dot{x}_2$ and $\xi_2 = -\frac{1}{2} x_1 \dot{x}_1$. We denote by $\omega = \arctan\left(\frac{B}{A}\right)$ the angle of the LRL vector with the first coordinate axis.

\begin{remark}
The existence of this conserved quantity is related to the fact that the three-dimensional Kepler problem possesses an $SO(4)$-symmetry, rather than just the obvious $SO(3)$-symmetry. In suitable coordinates a solution can be ``rotated'' into other solutions with the same energy but different angular momentum \cite{morehead2005visualizing,rogers1973symmetry}.
\end{remark}

\subsection{Precession in the perturbed Kepler problem}

Now consider the perturbed Kepler problem, $\cL = \frac{1}{2}|\dot{x}|^2 + \frac{1}{|x|} + \varepsilon \overline{\cL} (x,\dot{x})$. 
Note that this also induces a perturbation in the prolonged vector field, which now reads $\xi^{(1)} + \varepsilon \overline{\xi^{(1)}}$, because the quantities $\dot{\xi}_1$ and $\dot{\xi}_2$ contain second derivatives which are evaluated using the perturbed equations of motion. We call the change in angle of the LRL vector over one period of the unperturbed system the \emph{precession rate}.

\begin{proposition}\label{prop-precession}
If the major axis of an orbit is $\cO(\varepsilon)$-close to the $x_2$-axis, then the precession rate is
\begin{equation}\label{precession}
\Delta \omega = - \frac{2 \varepsilon T}{e} \average{ \inner{ \EL(\overline{\cL}) }{ \xi } } + \cO(\varepsilon^2),
\end{equation}
where $T$ is the period of the unperturbed orbit, $\EL(\overline{\cL}) = \der{\overline{\cL}}{x} - \frac{\D}{\D t} \der{\overline{\cL}}{\dot{x}}$ is the Euler-Lagrange expression for $\overline{\cL}$, $\xi = (\xi_1,\xi_2)$ is defined by Equation \eqref{defxi}, and $\average{\,\cdot\,}$ denotes the average over one period.
\end{proposition}

\begin{proof}
Set $G = \frac{x_1}{|x|}$ and $F = \overline{\xi^{(1)}} \cL + \xi^{(1)} \overline{\cL}$, then
\[ \left( \xi^{(1)} + \varepsilon \overline{\xi^{(1)}} \right) \left( \cL + \varepsilon \overline{\cL} \right) = \inner{\der{G}{\dot{x}}}{\dot{x}} + \varepsilon F + \cO(\varepsilon^2), \]
where $\xi^{(1)} + \varepsilon \overline{\xi^{(1)}}$ is the first prolongation of $\xi$ on solutions of the Euler Lagrange equations of the perturbed Lagrangian $\cL + \varepsilon \overline{\cL}$. Hence by Theorem \ref{thm-noether} it follows that
\[ \frac{\D}{\D t}\left( \inner{ \der{(\cL + \varepsilon \overline{\cL})}{\dot{x}} }{ \xi } - G \right) = \varepsilon F + \cO(\varepsilon^2), \]
from which we conclude that 
\begin{align}
\frac{\D A}{\D t} 
&= \varepsilon \left( F - \frac{\D}{\D t} \inner{ \der{\overline{\cL}}{\dot{x}} }{ \xi } \right) + \cO(\varepsilon^2) \notag\\
&= \varepsilon \left( \overline{\xi^{(1)}} \cL + \xi^{(1)} \overline{\cL} - \frac{\D}{\D t} \inner{ \der{\overline{\cL}}{\dot{x}} }{ \xi } \right) + \cO(\varepsilon^2). \label{precession-step1}
\end{align}
Now observe that
\begin{align*}
\xi^{(1)} \overline{\cL} - \frac{\D}{\D t} \inner{ \der{\overline{\cL}}{\dot{x}} }{ \xi }
&= \inner{ \der{\overline{\cL}}{x} }{ \xi } + \inner{ \der{\overline{\cL}}{\dot{x}} }{ \dot{\xi} } 
- \frac{\D}{\D t}  \inner{ \der{\overline{\cL}}{\dot{x}} }{ \xi } \\
&= \inner{ \EL(\overline{\cL}) }{ \xi } + \cO(\varepsilon),
\end{align*}
where the error term comes from the fact that $\dot{\xi}$ is evaluated on the unperturbed system. We also have that
\[
\overline{\xi^{(1)}}\cL
= \inner{ \der{\xi_1}{\dot{x}} }{ \EL(\overline{\cL}) } \dot{x}_1 +  \inner{ \der{\xi_2}{\dot{x}} }{ \EL(\overline{\cL}) } \dot{x}_2
= \inner{ \der{\xi_1}{\dot{x}} \dot{x}_1 + \der{\xi_2}{\dot{x}} \dot{x}_2 }{ \EL(\overline{\cL}) } .
\]
For our choice of $\xi$, defined in Equation \eqref{defxi}, we have $ \der{\xi_1}{\dot{x}} \dot{x}_1 + \der{\xi_2}{\dot{x}} \dot{x}_2 = (\xi_1,\xi_2) = \xi$, hence Equation \eqref{precession-step1} simplifies to
\[ \frac{\D A}{\D t} = 2 \varepsilon \inner{ \EL(\overline{\cL}) }{ \xi } + \cO(\varepsilon^2) . \]
The change in angle of the Laplace-Runge-Lenz vector is given by
\[ \dot{\omega} = \frac{\D}{\D t} \left( \arctan\frac{B}{A} \right) = \frac{1}{A^2 + B^2} \left( A \frac{\D B}{\D t} - B \frac{\D A}{\D t} \right). \]
Choose a coordinate system such that $A = \cO(\varepsilon)$ and $B \geq 0$. Then $B$ approximately equals the eccentricity $e$ and the derivative of the angle of the LRL vector is
\[ \dot{\omega} = -\frac{1}{B} \frac{\D A}{\D t} + \cO(\varepsilon^2)
= -\frac{2\varepsilon}{e} \inner{ \EL(\overline{\cL}) }{ \xi } + \cO(\varepsilon^2) . \qedhere \]
\end{proof}

\section{Numerical precession}

We now apply Proposition \ref{prop-precession} to the modified Lagrangians from Section\ \ref{sect-old}. This gives us a leading order estimate of the precession rates of the integrators.

\subsection{St\"ormer-Verlet scheme}

The perturbation term of the truncated modified Lagrangian \eqref{modlag-sv} is
\[ \varepsilon \overline{\cL} = \frac{h^2}{24} \left( \frac{1}{|x|^4} - 2 \frac{ |\dot{x}|^2 }{|x|^3} + 6 \frac{\inner{ x }{ \dot{x} }^2}{|x|^5} \right) . \]
In the following we identify $\varepsilon = \frac{h^2}{24}$. We want to evaluate Equation \eqref{precession}. Using the leading order equations of motion \eqref{eqsofmot} we find
\[ \EL(\overline{\cL}) = 4 \frac{x}{|x|^6} - 6 \frac{|\dot{x}|^2 x}{|x|^5} + 30 \frac{\inner{x}{\dot{x}}^2 x}{|x|^7} - 12 \frac{\inner{x}{\dot{x}} \dot{x}}{|x|^5} + \cO(h^2). \]
Using the fact that $\inner{x}{\xi} = \frac{1}{2} (x_1 \dot{x}_2 - \dot{x}_1 x_2) x_2 = \frac{1}{2} \mathds{L} x_2$ and $\inner{\dot{x}}{\xi} = \mathds{L} \dot{x}_2$, the leading order equations of motion, and Proposition \ref{prop-EandL} we obtain
\begin{align}
\average{ \inner{ \EL(\overline{\cL}) }{ \xi } }
&= \average{2 \frac{x_2}{|x|^6} - 3 \frac{|\dot{x}|^2 x_2}{|x|^5} + 15 \frac{\inner{x}{\dot{x}}^2 x_2}{|x|^7} - 12 \frac{\inner{x}{\dot{x}} \dot{x}_2}{|x|^5} } \mathds{L} + \cO(h^2) \notag \\
&= \average{30 \frac{x_2}{|x|^6} + 24 \mathds{E} \frac{x_2}{|x|^5} - 15 \mathds{L}^2 \frac{x_2}{|x|^7} + 4 \frac{\D}{\D t} \frac{\dot{x}_2}{|x|^3} } \mathds{L} + \cO(h^2) . \label{before-lemma}
\end{align}
The average $[\cdot]$ is taken along the unperturbed orbit, which is periodic, so $\average{ \frac{\D}{\D t}\frac{\dot{x}_2}{|x|^3} } = 0$. For the other terms we have the following Lemma.
\begin{lemma}\label{lemma-averages}
On solutions of the unperturbed Kepler problem for which the major axis is the $x_2$-axis there holds
\begin{enumerate}[$(a)$]
\item $\displaystyle \average{\frac{x_2}{|x|^5}} = \frac{a}{b^5} e$,
\item $\displaystyle \average{\frac{x_2}{|x|^6}} = \frac{a^2}{b^7} \left( \frac{3}{2} e + \frac{3}{8} e^3 \right)$,
\item $\displaystyle \average{\frac{x_2}{|x|^7}} = \frac{a^3}{b^9} \left( 2 e + \frac{3}{2} e^3 \right)$,
\end{enumerate}
where $a$ and $b$ are the semimajor and semiminor axes of the orbit respectively, and $e$ is the eccentricity.
\end{lemma}
\begin{proof}
Introduce polar coordinates $x_1 = - r \sin \theta$, $x_2 = r \cos \theta$, where $\theta=0$ corresponds to the positive $x_2$-axis. We have
\[  \average{\frac{x_2}{|x|^k}} = \average{\frac{cos \theta}{|x|^{k-1}}} 
= \frac{1}{T} \int_0^T \frac{cos \theta}{|x|^{k-1}} \,\D t . \]
Using Kepler's laws as in \cite{curtis1987expectation} we can rewrite this as
\begin{align*}
 \average{\frac{x_2}{|x|^k}} 
&= \frac{b^{5-2k}}{\pi a^{4-k}} \int_0^\pi (1 + e \cos\theta)^{k-3} \cos\theta \,\D \theta \\
&= \frac{b^{5-2k}}{\pi a^{4-k}} \int_0^\pi \sum_j \binom{k-3}{j} e^j \cos^{j+1} \theta \,\D \theta.
\end{align*}
Whenever, $j$ is even, we have $\int_0^\pi \cos^{j+1} \theta \,\D \theta = 0$. For $j = 1$ and $j = 3$ we find $\int_0^\pi \cos^2 \theta \,\D \theta = \frac{\pi}{2}$ and $\int_0^\pi \cos^4 \theta \,\D \theta = \frac{3 \pi}{8}$. Hence
\[ \average{\frac{x_2}{|x|^k}} = 
 \frac{b^{5-2k}}{\pi a^{4-k}} \left( \frac{\pi}{2}\binom{k-3}{1}e + \frac{3 \pi}{8}\binom{k-3}{3} e^3 + \ldots \right).\]
The claims now follow by evaluating this expression for $k=5,6,7$.
\end{proof}

Combining Proposition \ref{prop-precession}, Eq.\ \ref{before-lemma}, and Lemma \ref{lemma-averages} we find that the precession per revolution is given by
\begin{align*}
& -4 \pi a^{3/2}  \frac{h^2}{24} \left( 30 \frac{a^2}{b^7} \left( \frac{3}{2} + \frac{3}{8} e^2 \right) + 24 \frac{-1}{2a} \frac{a}{b^5} - 15 \frac{b^2}{a} \frac{a^3}{b^9} \left( 2 + \frac{3}{2} e^2 \right) \right) \frac{b}{\sqrt{a}} \sgn(\mathds{L}) + \cO(h^4)\\
&\quad = -4 \pi ab \frac{h^2}{24} \left( 30 \frac{a^2}{b^7} \left( \frac{15}{8} - \frac{3}{8} \frac{b^2}{a^2} \right) + 24 \frac{-1}{2a} \frac{a}{b^5} - 15 \frac{b^2}{a} \frac{a^3}{b^9} \left( \frac{7}{2} - \frac{3}{2} \frac{b^2}{a^2} \right) \right) \sgn(\mathds{L}) + \cO(h^4) \\
&\quad = -\frac{\pi h^2}{24} \left( 15 \frac{a^3}{b^6} - 3 \frac{a}{b^4} \right) \sgn(\mathds{L}) + \cO(h^4) ,
\end{align*}
assuming the major axis of the orbit is $\cO(h^2)$-close to the $x_2$-axis. However, since both this expression and the perturbed Kepler problem are rotationally symmetric, we can conclude that statement holds regardless of the orientation of the major axis.

In summary we have the following:

\begin{theorem}\label{thm-sv}
The numerical precession rate of the St\"ormer-Verlet method with step size $h$ is 
\[ -\sgn(\mathds{L}) \frac{\pi}{24} \left(15 \frac{a^3}{b^6} - 3 \frac{a}{b^4} \right)h^2 + \cO(h^4), \]
where $a$ and $b$ denote the semimajor and semiminor axes of the orbit of the exact solution.
\end{theorem}

For the example shown in Figure\ \ref{fig-sv}, the precession rate predicted by Theorem \ref{thm-sv} is $0.067$ radians per revolution and the observed numerical precession rate is $0.064$ radians per revolution.

\subsection{Implicit midpoint rule}

In exactly the same way as for the St\"ormer-Verlet method, we obtain the following result:

\begin{theorem}\label{thm-mp}
The numerical precession rate of the midpoint rule with step size $h$ is 
\[ \sgn(\mathds{L}) \frac{\pi}{12} \left(15 \frac{a^3}{b^6} - 3 \frac{a}{b^4} \right)h^2 + \cO(h^4). \]
\end{theorem}

Note that this expression differs by exactly a factor $-2$ from the expression for the St\"ormer-Verlet method. We will exploit this in the next section to construct new integrators.

For the example shown in Figure\ \ref{fig-mp}, the precession rate predicted by Theorem \ref{thm-mp} is $-0.13$ radians per revolution and the observed numerical precession rate is $-0.16$ radians per revolution.

\section{New integrators}\label{sect-new}

Based on Theorems \ref{thm-sv} and \ref{thm-mp} we propose three new integrators. They all have a precession rate of order $\cO(h^4)$ instead of $\cO(h^2)$.

\subsection{Linear combination of the Lagrangians}

Consider the discrete Lagrangian
\begin{align*}
 L(x_j,x_{j+1}) &= \frac{2}{3} L_{SV}(x_j,x_{j+1}) + \frac{1}{3} L_{MP}(x_j,x_{j+1}) \\
&=  \frac{1}{2} \left| \frac{x_{j+1}-x_j}{h} \right|^2 - \frac{1}{3} U\left( x_{j} \right) - \frac{1}{3} U\left( x_{j+1} \right) - \frac{1}{3} U\left( \frac{x_{j} + x_{j+1}}{2} \right).
\end{align*}
Its Euler-Lagrange equations define an implicit method,
\[ x_{j+1} - 2 x_j + x_{j-1} = - \frac{2 h^2}{3} U'(x_j) - \frac{h^2}{6} U'\left( \frac{x_{j-1} + x_{j}}{2} \right) - \frac{h^2}{6} U'\left( \frac{x_{j} + x_{j+1}}{2} \right). \]
We refer to this integrator as the mixed Lagrangian (ML) method.

\subsection{Lagrangian Composition}

Consider the discrete Lagrangians
\[ L_j(x_k,x_{k+1}) = \begin{cases}
L_{MP}(x_k,x_{k+1}) = \frac{1}{2} \left| \frac{x_{k+1}-x_k}{h} \right|^2 - U\left(\frac{x_k + x_{k+1}}{2}\right) & \text{if } 3|j , \\
L_{SV}(x_k,x_{k+1}) = \frac{1}{2} \left|\frac{x_{k+1}-x_k}{h}\right|^2 - \frac{1}{2} U(x_k) - \frac{1}{2} U(x_{k+1}) & \text{otherwise.}
\end{cases} \]
We look for a discrete curve $(x_j)_j$ that extremizes the action
\[ \sum_{j=1}^N L_j(x_{j-1},x_j) = L_{SV}(x_0,x_1) + L_{SV}(x_1,x_2) + L_{MP}(x_2,x_3) + \ldots . \]
This gives us three different Euler-Lagrange equations which are applied for different values of $j$\ mod\ 3. Indeed $\mathrm{D}_2 L_j(x_{j-1},x_j) + \mathrm{D}_1 L_{j+1}(x_j,x_{j+1})$ simplifies to
\[ \begin{cases}
\displaystyle x_{j+1}-2x_j+x_{j-1} = -\frac{h^2}{2} U'\left(\frac{x_{j-1}+x_j}{2}\right) - \frac{h^2}{2} U'(x_j) & \qquad\text{if } j \equiv 0 \mod 3 , \\
\displaystyle x_{j+1}-2x_j+x_{j-1} = - h^2 U'(x_j) & \qquad\text{if } j \equiv 1 \mod 3 , \\
\displaystyle x_{j+1}-2x_j+x_{j-1} = -\frac{h^2}{2} U'\left(\frac{x_j+x_{j+1}}{2}\right) - \frac{h^2}{2} U'(x_j) & \qquad\text{if } j \equiv 2 \mod 3 .
\end{cases} \]
Hence to determine the evolution we alternate between the St\"ormer-Verlet method (for $j \equiv 1$\ mod\ 3) and two new difference equations. We refer to this integrator as the Lagrangian composition (LC) method. Strictly speaking the LC method should be considered as an integrator with step size $3h$, but for fair comparison with the other methods we will still refer to the internal step $h$ as the step size.

This method of composing variational integrators is equivalent to composing the corresponding symplectic maps \cite[Section\ 2.5]{marsden2001discrete}.

\subsection{Composition of the difference equations}

Alternatively we can compose the difference equations obtained by the implicit midpoint rule and the St\"ormer-Verlet method respectively,
\begin{equation*}
\begin{cases}
\displaystyle x_{j+1}-2x_j+x_{j-1} = -\frac{h^2}{2} U'\left(\frac{x_{j-1}+x_j}{2}\right) -\frac{h^2}{2} U'\left(\frac{x_j+x_{j+1}}{2}\right) & \text{if } j \equiv 2 \ \text{ mod } 3,  \\
\displaystyle x_{j+1}-2x_j+x_{j-1} = - h^2 U'(x_j) & \text{otherwise.}
\end{cases}
\end{equation*}
We refer to this integrator as the difference equation composition (DEC) method. As for the LC method, we will abuse terminology and call the internal step $h$ the step size.

It is not clear if this construction yields a variational method, but numerical experiments show long-term near-conservation of energy and angular momentum. This seems to be a general phenomenon: also for other potentials $U$ and other variational integrators, the corresponding DEC method shows the long-term behavior one expects from a variational integrator.

\section{Numerical results}

In this section we compare the new methods of Section\ \ref{sect-new} numerically with the St\"ormer-Verlet scheme, the implicit midpoint rule, and two fourth order symplectic methods: the well-known integrator of Forest and Ruth \cite{forest1989fourth} and Chin's ``C'' algorithm which is especially well-suited for the Kepler problem \cite{chin1997symplectic,chin2000higher}.

\subsection{Choice of initial values}\label{sect-init}

In all our examples we use the initial values
\[ x(0) = (-3,0) \qquad \text{and} \qquad \dot{x}(0) = (0,0.45) . \]
For the discretizations we need specify $x_0 = x(0)$ and $x_1 \approx x(h)$. Our convention is to choose $x_1$ such that the discrete momentum $p_0 = -\mathrm{D}_1 L(x_0,x_1)$ equals the  initial velocity $\dot{x}(0)$. 

For the composition of difference equations no discrete Lagrangian and hence no discrete momentum is known. To determine the second initial point $x_1$ in this case we use the momentum $p_0$ corresponding to the St\"ormer-Verlet method, because this is the method we would have used to calculate $x_1$ if $x_0$ was not the first point.

The choice of the initial value $x_1$ does not affect the precession behavior. However, it can have a significant effect on the error over time. If the initial condition has a slightly wrong energy, then the period of the numerical solution will have a slight error as well. This will cause a linearly growing phase shift.

\subsection{Precession}\label{sect-experiment}

Figure\ \ref{fig-angle} shows the precession rates on a logarithmic scale for all five methods and a few choices of step size. It shows that the precession rates of the new methods behave like $h^4$, compared to $h^2$ for the methods from Section\ \ref{sect-old}. 

As for the three new methods, the mixed Lagrangian method beats the Lagrangian composition method, but the surprising winner is the composition of difference equations.

All our new methods have smaller precession rates than the fourth order symplectic integrator of Forest and Ruth \cite{forest1989fourth}. On the other hand,  Chin's fourth order symplectic ``C'' algorithm \cite{chin1997symplectic,chin2000higher} outperforms our methods.

\begin{figure}[h]
\centering
\includegraphics[width=\linewidth]{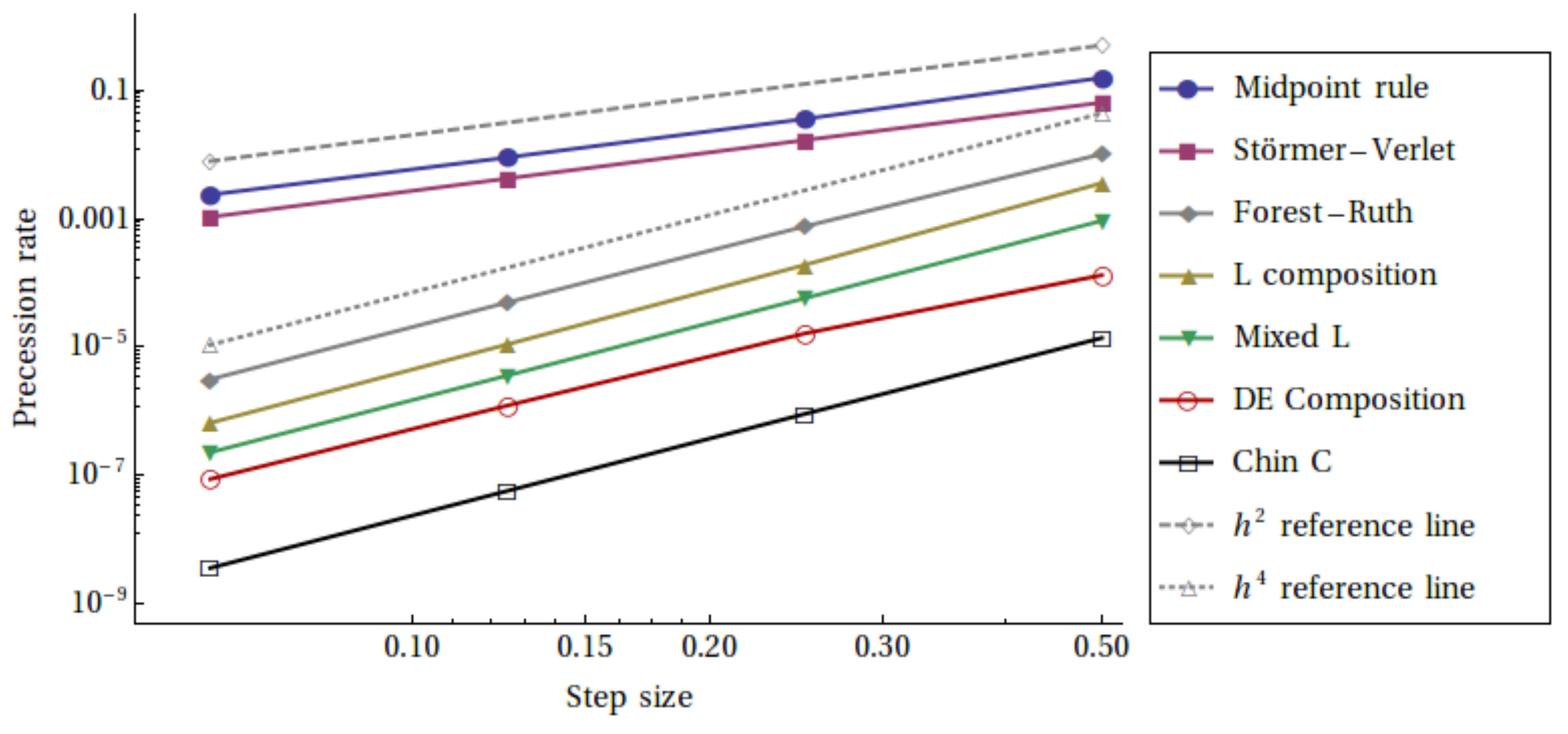}
\caption{Precession rate in radians per revolution for the different methods with step sizes $h=0.0625$, $h=0.125$, $h=0.25$ and $h=0.5$.}\label{fig-angle}
\end{figure}

\subsection{Total error}

The precession rate is not as closely related to the total error as one might expect. In many cases the numerical solution has a phase shift which contributes significantly to the total error. For the composition methods LC and DEC this phase shift is highly dependent on the step size and the initial conditions. Hence the total error growth for these methods is also sensitive to the choice of step size and initial conditions. This can be seen by comparing Figure\ \ref{fig-error45} and Figure\ \ref{fig-error50}.

\begin{figure}[h]
\centering
\includegraphics[width=\linewidth]{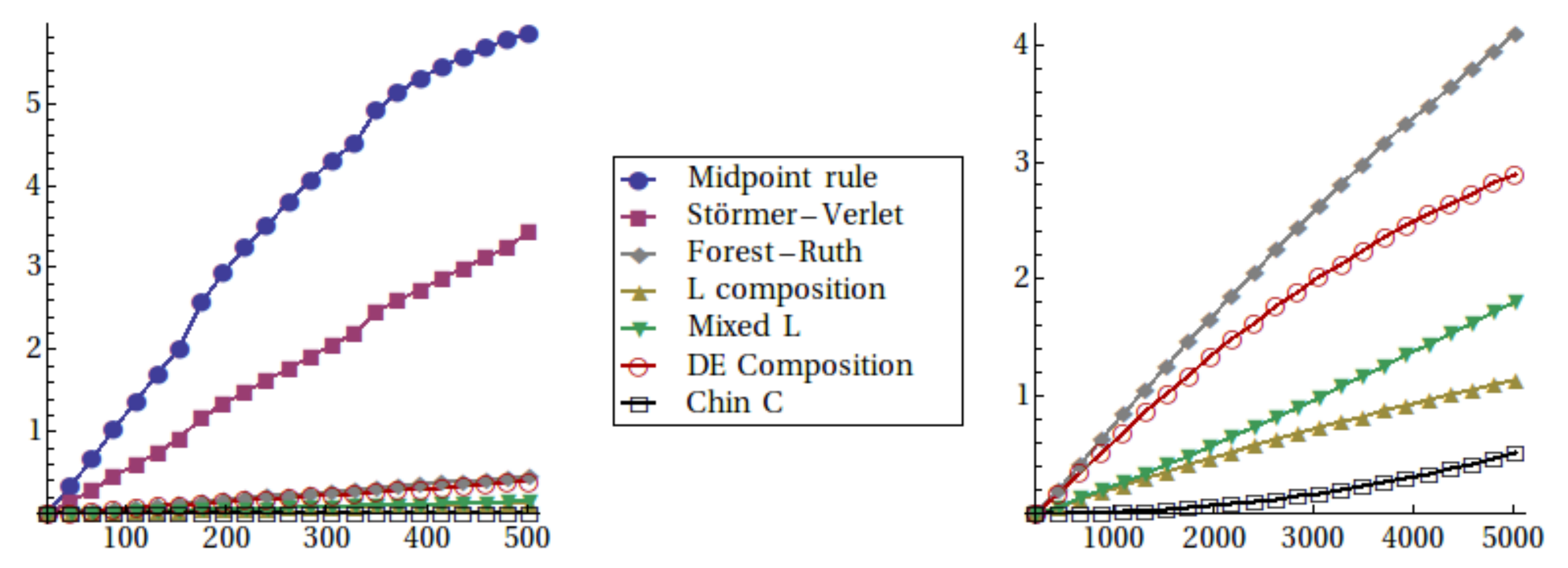}
\caption{The evolution of the error in position over a time interval of length $500$ (left) and $5\,000$ (right) with step size and $h=0.45$.}\label{fig-error45}
\end{figure}
\begin{figure}[h]
\centering
\includegraphics[width=\linewidth]{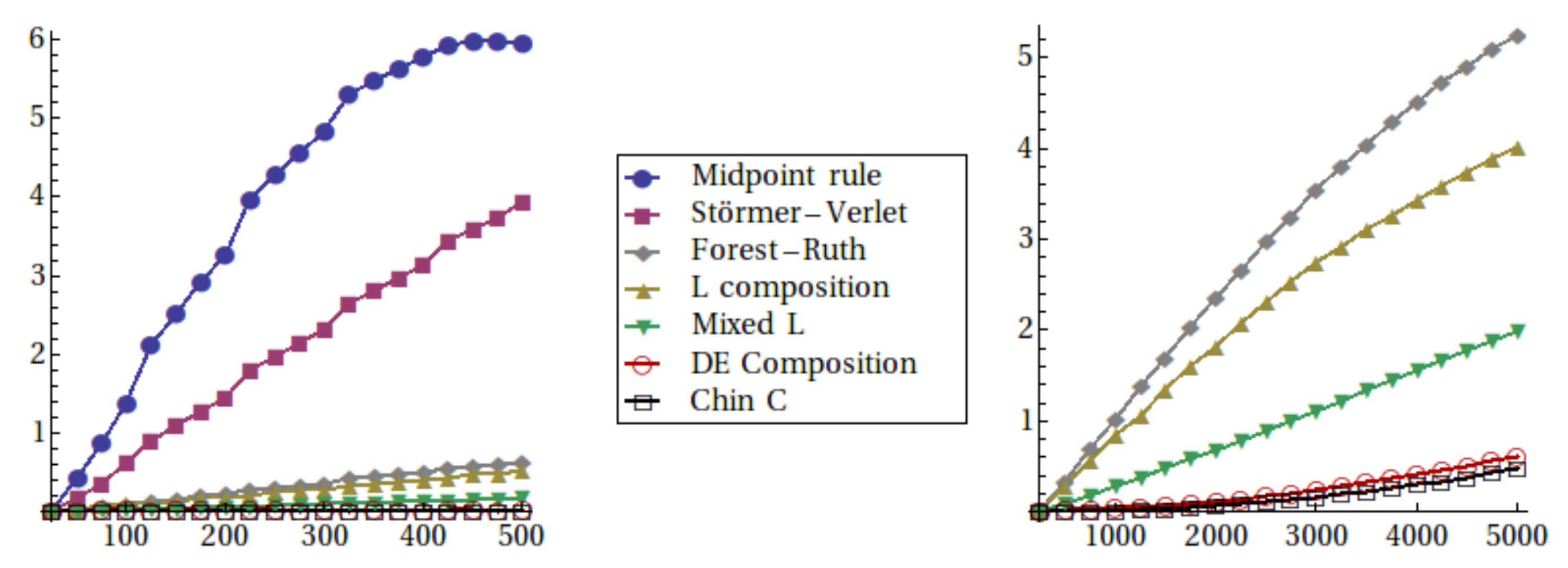}
\caption{The evolution of the error with step size $h=0.5$.}\label{fig-error50}
\end{figure}

\subsection{Speed}

To give a rough comparison of the computational effort required for the different methods, we list the running times of a long time calculation (20\,000 steps):
\begin{table}[h]
\centering
\begin{tabular}{rlr|rlr}
St\"ormer-Verlet & (SV) & 0.67s \quad & Mixed Lagrangian & (ML) & 23s \\
MidPoint rule & (MP) \quad & 22s \quad & \, Difference Equation composition & (DEC) \quad & 7.9s \\
Forest-Ruth & (FR) & 2.0s \quad & Lagrangian Composition & (LC) & 8.2s \\
Chin C & (C) & 2.2s \quad \\
\end{tabular}%
\end{table}%
\newline
We made a limited effort towards optimizing our implementation, so the given running times should only be taken as a rough indication. As expected the explicit methods SV, FR, and C are the fastest. For the composition methods DEC and LC only one out of every three steps is implicit, hence they are significantly faster than MP and ML.

\bigskip
\section*{Acknowledgements}

\small
This research is supported by the DFG Collaborative Research Center TRR 109 ``Discretization in Geometry and Dynamics''.

\bibliographystyle{abbrv}
\bibliography{kepler}

\begin{thebibliography}{10}

\bibitem{chin1997symplectic}
S.~A. Chin.
\newblock Symplectic integrators from composite operator factorizations.
\newblock {\em Physics Letters A}, 226(6):344--348, 1997.

\bibitem{chin2007physics}
S.~A. Chin.
\newblock Physics of symplectic integrators: Perihelion advances and symplectic
  corrector algorithms.
\newblock {\em Physical Review E}, 75(3):036701, 2007.

\bibitem{chin2000higher}
S.~A. Chin and D.~W. Kidwell.
\newblock Higher-order force gradient symplectic algorithms.
\newblock {\em Physical Review E}, 62(6):8746, 2000.

\bibitem{curtis1987expectation}
L.~J. Curtis, R.~R. Haar, and M.~Kummer.
\newblock An expectation value formulation of the perturbed {K}epler problem.
\newblock {\em Am. J. Phys}, 55(7):627--631, 1987.

\bibitem{forest1989fourth}
E.~Forest and R.~D. Ruth.
\newblock Fourth order symplectic integration.
\newblock {\em Physica}, 43(LBL-27662):105--117, 1989.

\bibitem{goldstein1980classical}
H.~Goldstein.
\newblock {\em Classical mechanics}.
\newblock Addison-Wesley Pub. Co., 1980.

\bibitem{hairer2006geometric}
E.~Hairer, C.~Lubich, and G.~Wanner.
\newblock {\em Geometric numerical integration: structure-preserving algorithms
  for ordinary differential equations}, volume~31.
\newblock Springer, 2006.

\bibitem{levy1971conservation}
J.-M. L{\'e}vy-Leblond.
\newblock Conservation laws for gauge-variant {L}agrangians in classical
  mechanics.
\newblock {\em American Journal of Physics}, 39(5):502--506, 1971.

\bibitem{marsden2001discrete}
J.~E. Marsden and M.~West.
\newblock Discrete mechanics and variational integrators.
\newblock {\em Acta Numerica 2001}, 10:357--514, 2001.

\bibitem{morehead2005visualizing}
J.~Morehead.
\newblock Visualizing the extra symmetry of the {K}epler problem.
\newblock {\em American journal of physics}, 73(3):234--239, 2005.

\bibitem{noether1918invariante}
E.~Noether.
\newblock Invariante {V}ariationsprobleme.
\newblock {\em Nachrichten von der Gesellschaft der Wissenschaften zu
  G{\"o}ttingen, mathematisch-physikalische Klasse}, 1918:235--257, 1918.

\bibitem{olver2000applications}
P.~J. Olver.
\newblock {\em Applications of {L}ie groups to differential equations}, volume
  107.
\newblock Springer Science \& Business Media, 2000.

\bibitem{rogers1973symmetry}
H.~H. Rogers.
\newblock Symmetry transformations of the classical {K}epler problem.
\newblock {\em Journal of Mathematical Physics}, 14(8):1125--1129, 1973.

\bibitem{vermeeren2015modified}
M.~Vermeeren.
\newblock Modified equations for variational integrators.
\newblock {\em \url{arXiv:1505.05411}}, 2015.

\bibitem{will1981theory}
C.~M. Will.
\newblock {\em Theory and experiment in gravitational physics}, volume~1.
\newblock Cambridge University Press, 1981.

\end{thebibliography}

\end{document}